\documentclass[12pt, a4paper]{amsart}
\usepackage[utf8]{inputenc}
\usepackage{latexsym,amsmath,amssymb,amsbsy,ulem,amsthm, amsfonts,
hyperref,amscd}

\theoremstyle {plain}
\newtheorem{theorem}{Theorem}
\newtheorem{lem}{Lemma}
\newtheorem{defe}{Definition}
\newtheorem{conj}{Conjecture}

\newcommand{\eps}{\varepsilon}

\newcommand{\Ent}[2]{\mathbb{H}_{#1}(#2)}
\def\diam{\rm{diam}}

\begin{document}

\title{Correction of metrics}
\author{P. B. Zatitskiy}
\address{St. Petersburg Department \\
of V.A.Steklov    Institute   of   Mathematics \\
of   the   Russian   Academy   of   Sciences
\\
Fontanka 27, \\ 191023, St. Petersburg, Russia
\\
P. L. Chebyshev Laboratory \\ of St. Petersburg State University
\\ 14th Line 29B, Vasilyevsky Island,\\
199178, St. Petersburg, Russia}
 \email{paxa239@yandex.ru}

\author{F. V. Petrov}
\address{St. Petersburg Department \\
of V.A.Steklov    Institute   of   Mathematics \\
of   the   Russian   Academy   of   Sciences
\\
Fontanka 27, \\ 191023, St. Petersburg, Russia}
\email{fedyapetrov@gmail.com}

\date{1 Oct. 2011}

\thanks{This research is supported by the 
Chebyshev Laboratory (Department of Mathematics and 
Mechanics, St. Petersburg State University) under 
RF Government grant 11.G34.31.0026, and RFBR grants
11-01-12092-ofi-m and
11-01-00677-a.}
\maketitle

Metric triple $(X,\rho,\mu)$ is a space with structures of metric $\rho$
and Borel probabilistic measure $\mu$. The studying of them was initiated by
M.~L.~Gromov in \cite{G}. Namely, it was proved that metric triples
are uniquely parametrized by measures on the set of distance matrices
$(\rho(x_i,x_j))_{i,j}$, generated by random and independent 
choice of points
 $x_1,x_2,\dots\in X$. A.~M.~Vershik gave another formulation of this fact
and  proposed its new proof, based on the Ergodic theorem, and initiated
the studying of variable metrics on the measurable space. In particular, in
 \cite{AMU,AM} he suggested to use dynamics of metrics for studying the
 entropy and other invariants of dynamical systems. It requires the preliminary
 studying of the set of matrics and ``almost metrics'' on the standard
 measurable space. The present note is devoted to correction
 of almost semimetrics to real metrics on the set of
 full measure. We  are going to
 use the obtained theorems in further
work on application of metrics dynamics to ergodic theory.

In this article we will consider only Lebesgue spaces with continuous measure.

\begin{defe} Let $(X,\mu)$ be Lebesgue space with probabilistic measure,
$\rho(x,y)$ be measurable non-negative function on $(X\times
X,\mu\times \mu)$ such that $\rho(x,y)=\rho(y,x)$
and $\rho(x,z)\leq \rho(x,y)+\rho(y,z)$ for almost all $x$, $y$, $z$
in $X$. We say that $\rho$ is \emph{almost-metric} on $X$.

We say that almost-metric  $\rho$ is \emph{essentially separable},
if for any $\varepsilon>0$ one may cover
$X$ by a countable set of measurable
subsets of essential diameter less then
$\varepsilon$ (essential diameter of the set
$A\subset X$ is defined as the essential supremum
of the function $\rho(x,y)$, restricted to
 $A\times A$).

Two almost-metrics which coincide on the set of full $(\mu\times \mu)$-measure as
functions of 2 variables, are called \emph{equivalent}.
Sometimes we also say that one of two
equivalent metrics is a \emph{correction} of another
(usually corrected metric is in a  sense better then the initial one).

Metric (or semimetric) on the Lebesgue space is called \emph{admissible},
if the corresponding (semi-)metric space is separable on the set
of full measure.

We say that a semimetric
$\rho$ on the space with Borel probabilistic
measure has a finite  $\eps$-entropy $\Ent{\eps}{\rho}<\infty$,
if there exist a finite number of balls of radius $\eps$,
which cover the set of measure at least  $1-\eps$.

Closed ball of radius $r$ centered in $x$ is denoted by $B(x,r)$.

\end{defe}

The following lemma lists several equivalent definitions
of admissible metric.

\begin{lem}\label{kritdop} Let $(X,\rho)$ be a semimetric space
and $\mu$ be Borel measure on it. Then the following statements are equivalent:\\
{\rm1)} For any $\eps>0$ the semimetric $\rho$ has finite  $\eps$-entropy.\\
{\rm 2)} There exists a set of full measure $\mu$ such that restriction of $\rho$ on this set is separable\\
{\rm 3)} For $\mu$-almost all $x$ all balls (in semimetric $\rho$) centered in $x$ have positive measure.\\
\end{lem}

\begin{proof}
Let's prove implication from 1) to 3). Consider the sets
 $T_n=\{x\in X: \mu(B(x,\frac{1}{n}))=0\}.$ It suffices
 to prove that $\mu(T_n)=0$ for all $n$. A priori we do not even know
 that $T_n$ is measurable, but we do not need it.
  Choose any $\eps \in (0,\frac{1}{10n})$. Condition 1)
  says that $X$ may be partitioned onto disjoint sets
  $X_0,X_1,\dots,X_k$ such that $\mu(X_0)<\eps$ and $\diam(X_j)<2\eps$ for all $j\in\{1,\dots,k\}$.
  Without loss of generality sets $X_j$ have positive measure (if $\mu(X_j)=0$, remove it
  and replace $X_0$ to $X_0\cup X_j$). Then for any point
$x \notin X_0$ the ball $B(x,2\eps)$ contains one of the sets $X_j$,
hence has strictly positive measure. It follows that $T_n \subset X_0$.
So we have proved that for arbitrarily small $\eps>0$ the set
$T_n$ is contained in a set of measure less then
 $\eps$. It follows that $T_n$
 is measurable and have zero measure. The union of $T_n$'s taken by all positive integers $n$
 has zero measure aswell. This proves 3).

Let' prove that 3) implies 2). Consider  the set $Y$ of points such that all balls
centered in those points have positive measure. Then $Y$ is a set of full measure and it suffice
to prove that for any $\eps>0$ it has a countable $\eps$-net.
Assume that for some $\eps>0$ it does not have a countable $\eps$-net, then by Zorn's lemma
$Y$ has a more then countable subset with mutual distances at least $\eps$. But then
the balls of radius $\eps/3$ centered in those points  have
positive measures and are mutually disjoint, so the sum of their measures does not exceed 1. But
any summable family is at most countable. A contradiction.

Let' finally prove that 2) implies 1). Let a countable set $\{x_n\}_{n=1}^{\infty}$
be dense in a set $X'$ of full measure. For fixed $\eps>0$ consider the balls
$B_n=B(x_n,\eps)$. The union of them has full measure. But then for some finite $N$ the finite
union $\bigcup\limits_{n=1}^N B_n$ has measure more then $1-\eps$, as desired.
\end{proof}

\begin{theorem}\label{ispravlenie}

{\rm1)} Let $(X,\mu)$ be a Lebesgue space, $\rho$ be almost metric on $X$.
Then $\rho$ may be corrected to everywhere finite semimetric
on $X$.

{\rm2)} If $\rho$ was essentially separable, then corrected
semimetric may be chosen so that  $(X,\rho)$ is separable.
In other words, the corrected semimetric may be chosen
admissible.
\end{theorem}

\begin{proof}
1) Note that for almost all $x$ the function
$\rho(x,\cdot)$ is measurable and inequality
$\rho(x,y)+\rho(x,z)\geq \rho(y,z)$ holds for almost all pairs
$(y,z)$. Fix such point $x_0$. At first, we replace the measure $\mu$
to some quivalent measure so that the function $f(t)=\rho(x_0,t)$ becomes summable.
For example, we may take $A_n=f^{-1}([n-1,n))$ and put
$\tilde{\mu}(B)=c\sum_{n=1}^{\infty} 2^{-n} \mu(B\cap A_n)$
for any measurable $B$ and the constant $c$ is chosen
so that
$\tilde{\mu}(X)=1$. Note that now the function
$\rho(y,z)$ becomes summable as a function of two variables by
triangle inequality.

Now we identify Lebesgue space and the unit circle $S=\mathbb{R}/\mathbb{Z}$
equipped byt he Lebesgue probabilistic measure.
Define the new metric
$\tilde{\rho}$ (possibly infinite somewhere) by equality
$$
\tilde{\rho}(x,y)=\limsup_{T\rightarrow +0} T^{-2}\int_{0}^T\int_0^T \rho(x+t, y+s) dt ds \eqno(1)
$$

Note that by Lebesgue theorem on differentiation
of integral the limit exists and equals $\rho(x,y)$ almost
surely (i.e. for almost all pairs $(x,y)$).
This function $\tilde{\rho}$ is obviously symmetric. Let's prove triangle inequality
for it. Note that for almost all
$(s,t,\tau)\in [0,T]^3$ one has
$$\rho(y+s, z+t)\leq \rho(y+s, x+\tau)+\rho(x+\tau, z+t).$$
Let's integrate it by
 $[0,T]^3$, divide by $T^3$ and take an upper limit for $T\rightarrow +0$.
Using inequality $\limsup (F+G)\leq \limsup F+\limsup G$ we get a triangle
inequality for $\tilde{\rho}$. If $\tilde{\rho}$ does not vanish in some points
on diagonal, replace those values to 0. Note that the function
$\tilde{\rho}$ is finite almost everywhere (since
equivalnet function $\rho$ is almost everywhere finite).
Hence one may choose a point $x_0$ so that $\tilde{\rho}(x_0,x)<\infty$ for
almost all
$x\in S$, i.e. for all $x$ in the set $S_1$ of full measure.
Now change the semimetric $\tilde{\rho}$ outside $S_1\times S_1$ by the rules
$\tilde{\rho}(x,y):=\tilde{\rho}(x_0,y)$ for $x\notin S_1, y \in S_1$,
$\tilde{\rho}(x,y):=\tilde{\rho}(x,x_0)$ for $y\notin S_1, x \in S_1$ and
$\tilde{\rho}(x,y):=0$ for $x,y\notin S_1$. This semimetric
is almost everywhere finite.

2) Using the statement of p. 1), we may
suppose that the semimetric $\rho$ is defined on whole $X$ and
satisfies triangle inequality everywhere. Let's prove that there
exists a set $Y$ of full measure so that restriction of $\rho$
on $Y$ is a separable semimetric space.
The correction outside $Y$ is done as descibed above in the end
of proof of 1), and separability is saved after such correction.
It suffices for any fixed positive integer $n$ to find a countable union of balls
of radius $1/n$ which has a full measure, then define $Y$ as the intersection
of such sets. Let's cover $X$ by a countable number of measurable sets
of diameter at most $1/n$. Consider one such set $A$. For almost
all points $x\in A$ the distances from $x$ to almost
all points of $A$ do not exceed  $1/n$. Hence $A$ mod $0$ is covered
by a ball of radius $1/n$, and we are done.
\end{proof}

Usually when one says about the space with metric
and measure, she considers the metric structure as ``main'',
and restricts measure to satisfy some properties in terms
of metric (Borel, regular measures). We follow
A.~M.~Vershik's approach, considering the metric as
measurable function. Note that if $(X,\rho)$
is separable semimetric space, and $\mu$ is a Borel
measure on $X$, then $\rho$ as a function on
$X\times X$ is measurable w.r.t. the measure $\mu\times \mu$
(since it is continuous and hence Borel
measurable).

The follwoing theorem
shows that in separable case
those two conditions
(``measure is Borel'' and ``metric is measurable'')
are equivalent.


\begin{theorem} Let $(X,\mu)$ be the Lebesgue space,
$\rho$ be admissible semimetric on $X$. Then the measure
$\mu$ is Borel w.r.t. topology of metric $\rho$.
\end{theorem}

\begin{proof}
For any rational $r>0$ the set $\{(x,y):\,\rho(x,y)<r\}$
is measurable, hence almost all its sections
(balls of radius $r$, hereafter balls mean open balls).
Hence foralmost all ponts $x\in X$ all balls with rational radius
center in $x$ are measurable. In this case all balls centered in $x$ are measurable.
Denote by $X_1$ the set of such points, let $X_2$ be the complement of $X_1$
($\mu(X_2)=0$). Consider a countable dense subset $X'\subset X_1$. Let's prove
that any ball is measurable, from separability it follows that measure is Borel.
Consider the ball $B=B(x_0,r_0)$ centered in $x_0$ with
raidus $r_0$. For any point
 $x\in X'\cap B$ consider the ball $B(x,r_0-\rho(x,x_0))$.
 It lies in $B$ and is measurable. Let's prove that the union of such
 balls contains $X_1\cap B$. Then $B$ contains their
 union $U$ and is contained in $U\cup X_2$, hence $B$ is measurable
 and $B=U$ mod 0.  Take any point $x\in X_1\cap B$. Let's find a point $y\in X'$
such that $\rho(x,y)<(r_0-\rho(x_0,x))/2$. Then by triangle inequality
$\rho(x_0,y)<(r_0+\rho(x_0,x))/2$, so $y\in B$ and moreover the ball $B(y,r_0-\rho(x_0,y))$
contains a point $x$, as we wish.
\end{proof}

If the metric is not separable,
the conclusion of this theorem may fail. Indeed,
let $A$ be a non-measurable subset of $X$,
$x_0\in X\setminus A$ be a point,
define a metric
$$
\rho(x,y)=
\begin{cases}
0,&\text{if $x=y$;}\\
1,&\text{if $x=x_0$, $y\in A$ or $y=x_0$, $x\in A$;}\\
2,&\text{otherwise.}
\end{cases}
$$
It equals $2$ almost everywhere and is so measurable, but the ball $B(x_0,3/2)$
is not measurable.

One may ask an opposite question. Let be given some (not separable)
semimetric $\rho$ on the space $X$. Consider its Borel sigma-algebra
$\mathbb{B}$. Is it true that $\rho$, as a function on $X\times X$,
is measurable w.r.t. sigma-algebra $\mathbb{B} \times \mathbb{B}$?
We do not know the answer.

Now let's prove that two equivalent admissible metrics
coincide on the set of full measure. Of course, this
statement fails without admissibility condition (for example,
for metrics with distances values 1 and 2.)

\begin{theorem}
Let two admissible metrics $\rho_1,\rho_2$ coincide almost everywhere on
 $X\times X$ w.r.t. measure $\mu\times \mu$.
 Then there exists a subset $X'\subset X$ of full measure such that $\rho_1$ and $\rho_2$
 coincide on $X'\times X'$.
\end{theorem}

\begin{proof}
Note that for almost all $x\in X$, we have $\rho_1(x,y)=\rho_2(x,y)$ for almost all $y\in X$.
Removing the set of zero measure from $X$, we may suppose that it holds for all $x$.
Now for any $x\in X$ and any positive $r$ the balls $\{y\in X : \rho_1(x,y)<r\}$ and $\{y\in X : \rho_2(x,y)<r\}$
have equal measure. Using~\ref{kritdop} we may find
a subset $X'$ of full measure such that for any $x\in X'$ any
ball $\{y\in X : \rho_1(x,y)<r\}$ has a positive measure.
Then the same holds for  $\rho_2$. Let's prove that $\rho_1$ and $\rho_2$ coincide on $X'\times X'$
Take $x_1,x_2\in X'$. Take any $r>0$ and note
that for almost all $y$ such that $\rho_1(x_1,y)<r$ one has
 $\rho_1(x_1,y)=\rho_2(x_1,y)$ and $\rho_1(x_2,y)=\rho_2(x_2,y)$.
 Since the set of such $y$'s (it is a ball) has positive measure, we may
 find at least one point $y$ in it. Then
 $$\rho_2(x_1,x_2)\leq \rho_2(x_1,y)+\rho_2(y,x_2)=\rho_1(x_1,y)+\rho_1(y,x_2)\leq 2r+\rho_1(x_1,x_2).$$
 Since it holds for any $r>0$, we conclude that $\rho_2(x_1,x_2) \leq \rho_1(x_1,x_2)$. Opposite
 inequality is analogous. So, $\rho_1$ and $\rho_2$ coincide on $X'\times X'$ as desired.
\end{proof}

The natural question to ask is the following: which structures on measure spaces
(except semimetric space structure) admit correction theorems like Theorem \ref{ispravlenie}.
May one correct almost-group to the group, almost-vector space to vector-space
and so on? We do not know non-trivial examples, in which
the answer is negative (the trivial examples include, say, the observation
that almost-metric may be corrected only to semimetric, not to metric). We present another positive result

\begin{theorem}\label{ultra}
Let semimetric $\rho$ be defined on a Lebesgue space $X$, and let it
satisfy ultrametric inequality
$\rho(x,z)\leq \max(\rho(x,y),\rho(y,z))$
for almost all triples $(x,y,z)\in X^3$. Then there exists an ultrametric on $X$,
which coincides with $\rho$ on almost all pairs and satisfies ultrametric inequality for all triples.
\end{theorem}

\begin{proof}
For any  $n$ consider the correction of almost-metric
$\rho^n$ given by the formula
$$
(\rho_n(x,y))^n=\limsup_{T\rightarrow +0} T^{-2}\int_0^T\int_0^T \rho^n(x+t,y+s)dtds
$$

By power means inequalities the sequence $\rho_n(x,y)$
increases by $n$ for fixed $x, y$, so it has (finite or infnite) limit
$\tilde{\rho}(x; y)$. By Lebesgue's integrals differentiation theorem this
limit is almost everywhere finite and equal to $\rho(x,y)$.
The function $\tilde{\rho}^n$ is a semimetric for all $n$ (this is so for $\rho_m$
instead $\tilde{\rho}$ with $m\geq n$, and one may pass to limit in the
corresponding triangle inequality).
The infinite values may be avoided on the same way as in the proof of
Theorem~\ref{ispravlenie} for usual metrics.
\end{proof}

We also formulate the following general statement, which includes Theorems
~\ref{ispravlenie},~\ref{ultra} as partial cases.

\begin{conj} Given positive
integers $k\leq n$ and measurable real-valued function $f(x_1,x_2,\dots,x_k)$.
For almost all $y_1,\dots,y_n$ the vector $\{f(y_{i_1},\dots,y_{i_k})\}_{1\leq i_j\leq n}$
of dimension
$n^k$ belongs to the given closed subset in the space
of dimension $n^k$. Then there exists a function $\tilde{f}$, equivalent to $f$,
for which this condition holds for all $y_1,\dots,y_n$.
\end{conj}

The questions of this paper arised in the program in ergodic
theory, initiated by A.~M.~Vershik. We are also grateful to him for
support and numerous helpful discussions.

\bigskip

Zatitskiy P.~B., Petrov F.~V. Correction of metrics.

We prove that a symmetric nonnegative function of two
variables on a
Lebesgue space that satisfies the triangle inequality for
almost all triples
of points is equivalent to some semimetric. Some other
properties of metric
triples (spaces with structures of a measure space and a
metric space) are
discussed.


\begin{thebibliography}{99}
\bibitem{neumann} P.~R.~Halmos, J.~von Neumann,
{\it Operator methods in classical mechanics, II.} ---
 Ann. Math. \textbf{43}, No.~2, 332--350 (1942).

\bibitem{G} M.~Gromov, \emph{Metric Structures for Riemannian and Non-Riemannian Spaces.}
Birkhauser, 1998.

\bibitem{AMU} The universal Uryson space, Gromov's metric triples, and random 
metrics on the series of natural numbers.
Uspekhi Mat. Nauk \textbf{53}, No.~5, 57--64 (1998). English translation: 
Russian Math. Surveys \textbf{53}, No.~5, 921--928 (1998).


\bibitem{AM} A.~Vershik, {\it Scaling entropy and automorphisms with purely
point spectrum},
{\tt arXiv:1008.4946v4}.
\end{thebibliography}
\end{document}